\documentclass[oneside,english]{amsart}
\usepackage[latin9]{luainputenc}
\usepackage{amstext}
\usepackage{amsthm}
\usepackage{amssymb}
\usepackage{graphicx}

\makeatletter
\numberwithin{equation}{section}
\numberwithin{figure}{section}
\theoremstyle{plain}
\newtheorem{thm}{\protect\theoremname}
\theoremstyle{definition}
\newtheorem{defn}[thm]{\protect\definitionname}
\theoremstyle{plain}
\newtheorem{lem}[thm]{\protect\lemmaname}
\theoremstyle{plain}
\newtheorem{conjecture}[thm]{\protect\conjecturename}
\theoremstyle{plain}

\@ifundefined{showcaptionsetup}{}{%
 \PassOptionsToPackage{caption=false}{subfig}}
\usepackage{subfig}
\makeatother

\usepackage{babel}
\providecommand{\conjecturename}{Conjecture}
\providecommand{\definitionname}{Definition}
\providecommand{\lemmaname}{Lemma}
\providecommand{\propositionname}{Proposition}
\providecommand{\theoremname}{Theorem}

\begin{document}

\title{The centered convex body whose marginals have the heaviest tails}
\author{yam eitan}

\maketitle
\renewcommand{\partname}{Chapter}

\section*{abstract}

Given any real numbers $1<p<q$, we study the norm ratio (i.e. the
ratio between the $q$-norm and the $p$-norm) of marginals of \textit{centered}
convex bodies. We first show that some marginal of the simplex maximizes
said ratio in the class of $n$-dimensional centered convex bodies. We
then pass to the dimension independent (i.e. log-concave) case where
we find a 1-parameter family of random variables in which the maximum
ratio must be attained, and find the exact maximizer of the ratio
when $p=2$ and $q$ is even. In addition, we find another interesting
maximization property of marginals of the simplex involving functions
with positive third derivatives.

\section{introduction}

Let $1<p<q$, $K\subset\mathbb{R}^{n}$ be a convex body and $X\thicksim Uni(K)$
be a random vector uniformly distributed over K. From here forth we
shall assume K is centered (i.e. $\mathbb{E}[X]=0$). Define:
\[
\alpha(K):=\max_{\theta\in S^{n-1}}\frac{\lVert X\cdot\theta\rVert_{q}}{\lVert X\cdot\theta\rvert\rvert_{p}}.
\]
Here, $S^{n-1}$ is the $(n-1)-$dimensional unit sphere, centered
at the origin of $\mathbb{R}^{n}$. This quantity can be thought of
as a measurement of the ``heaviest tail'' of any 1-dimensional projection
of X. Our main result is the following:
\begin{thm}
Let $K\subset\mathbb{R}^{n}$ be a centered convex body and let $\Delta_{n}$
be an n-dimensional centered simplex. Then:
\begin{equation}
\alpha(K)\leq\alpha(\Delta_{n}).
\end{equation}
\end{thm}

While one might not be surprised that the simplex is a maximizer of
$\alpha(\cdot),$ as it is a maximizer of many other functionals over
the family of convex bodies (e.g. \cite{key-13}), the unit vector
for which the maximum ratio is achieved is not always one of the normals
of the simplex. In other words, for certain values of $p$ and $q$
there exist non-conic maximizers of $\alpha(\cdot)$. The relation
between $p,q,n,$ and the unit vector maximizing the norm ratio for
the simplex remains elusive thus far. Recall now the definition of
a random vector in isotropic position:
\begin{defn}
A random vector $X=(X_{1},\ldots,X_{n})$ is said to be in isotropic
position if:
\[
\mathbb{E}[X]=0,\hspace{3em}\mathbb{E}[X_{i}X_{j}]=\delta_{i,j},\hspace{3em}(\forall0\leq i,j\leq n),
\]
where $\delta_{i,j}$ is the Dirac delta. A convex body is said to
be in isotropic position if $X\sim Uni(K)$ is in isotropic position.
\end{defn}

Any convex body can be transformed into an isotropic one through an
affine map (see \cite{key-3} for proof). Our second main result is
the following:
\begin{thm}
Let $K\subset\mathbb{R}^{n}$ be a convex body in isotropic position
and $X\thicksim Uni(K)$ be a random variable uniformly distributed
over K. Then for any $k\in\{3,4\},$ $\phi\in C^{k}(\mathbb{R})$
such that $\phi^{(k)}(x)>0\ \forall x\in\mathbb{R}$, we have:
\begin{equation}
\max_{\theta\in S^{n-1}}\mathbb{E}[\phi(X\cdot\theta)]\leq\max_{\theta\in S^{n-1}}\mathbb{E}[\phi(\Gamma_{n}\cdot\theta)],
\end{equation}
where $\Gamma_{n}\thicksim Uni(\Delta_{n})$ is also in isotropic
position.
\end{thm}

Here if $k=3$, then the inequality is strict unless K is a cone,
in which case the maximizing unit vector is normal to (one of) the
cone's base(s). However, if $k=4$ this isn't the case, as for some
$\phi$ there exist non-conic maximizers. A version of Theorems 1
and 3 for the symmetric log-concave case can be found in Eskenazis,
Nayar, and Tkocs \cite{key-5}. In fact, we shall use techniques similar
to those presented in the said paper here. Recall that a random variable
$X$ is called log-concave if it has a density $f_{X}$ with respect
to the Lebesgue measure such that $log\left(f_{x}(\cdot)\right)$
is concave on the support of $f_{X},$ where $log$ is the natural
logarithm. The set of log-concave random variables can be defined
equivalently as the closure of the set of random variables of the
form $X\cdot\theta$ for all convex bodies (regardless of the dimension).
Thus our main result concerning the dimension independent case is
the following:
\begin{thm}
Let X be a centered log concave random variable, and let $n$ be an
even integer. We then have:
\begin{equation}
\frac{\lVert X\rVert_{n}}{\lVert X\lVert_{2}}\leq\left(n!\sum_{i=0}^{n}\frac{(-1)^{k}}{k!}\right)^{\frac{1}{n}}=\left(!n\right)^{\frac{1}{n}},
\end{equation}
with equality if and only if $X$ coincides with $\pm\Gamma$ in law.
Here, $!n$ is the subfactorial of $n$ and $\Gamma$ is a random
variable with density $g(x)=e^{-(x+1)}$ supported on the line $[-1,\infty)$.
\end{thm}

In addition, we prove the following result regarding odd moments of
log-concave random variables:
\begin{thm}
Let X be a centered log-concave random variable. If $q$ is an odd
integer, and $p$ is a real number such that $1<p<q,$ then we have:
\begin{equation}
\frac{\lvert\mathbb{E}[X^{q}]\rvert^{\frac{1}{q}}}{\lVert X\rVert_{p}}\leq\frac{\lvert\mathbb{E}[\Gamma^{q}]\rvert^{\frac{1}{q}}}{\lVert\Gamma\rVert_{p}},
\end{equation}
with equality if and only if $X$ coincides with $\Gamma$ in law
(The difference here is that the absolute value in the numerator is
taken outside of the expectation).
\end{thm}

A different proof of Theorem 5 for the case $p=2,\ q=3$ can be found
in Bubeck and Eldan \cite{key-4}. The problem of finding the maximum
norm ratio over different families of random variables has been discussed
in several other instances in the mathematical literature. The fact
that the norm ratios of marginals of convex bodies are bounded by
some universal constant was proven first by Berwald in \cite{key-1}
and then by Borell in \cite{key-2} (one can also find this result
in a survey \cite{key-10} by Milman and Pajor). Borel and Berwald
also found the maximum norm ratio of log-concave random variables
supported on the real half-line (see \cite{key-11} for proof). Another
example for a problem of a similar type is the Khintchine inequality,
which deals with the maximum norm ratio of linear combinations of
independent Bernoulli random variables. The sharp Khintchine inequality
for $p=2$ was established by Haagerup in \cite{key-6}. A simpler
proof by Nazarov and Podkorytov can be found in \cite{key-12}. The
case where $p\text{ and }q$ are even numbers was also solved. For
a more in-depth history of the Khintchine inequality as well as a
proof of the even case see \cite{key-11}. Eskenazis, Nayar, and Tkocs
also found the maximum norm ratio for the $s$-norm ball when $p=2$
in the aforementioned \cite{key-5}. Throughout this paper, by a convex
body we mean a convex compact set with non-empty interior, and by
a cone we mean the convex hull of a subset of a an affine space of
co-dimension one and a point outside of the said affine space.

\section*{acknowledgements}

This paper is a part of the author's master thesis under the suprevision of Professor Bo'az Klartag, and was partially funded by the Israel Science Foundation.

\section{the dimension dependent case}

Recall $X\sim Uni(K)$ for some n-dimensional, centered, convex body
$K$. We start by stating the well-known Brunn principle (see e.g.
\cite{key-14} for proof):
\begin{lem}
For any $\theta\in S^{n-1}$, the random variable $X\cdot\theta$
has a density $f_{\theta}(x)$ w.r.t. the Lebesgue measure. In addition
$f_{\theta}$ is supported on a compact interval and $f_{\theta}^{1/n-1}$
is concave in its support. Finally, if $f_{\theta}^{1/n-1}$ is affine
in its support, then $K$ is a (pherhaps truncated) cone and $\theta$
is normal to $K$'s base.
\end{lem}

Bruun's principle gives us valuable information about the way the
graph of $f_{\theta}$ can intersect the graphs of other certain functions
of interest. We shall make use of this information in conjunction
with the theory of Chebyshev systems (for an in-depth discussion on
Chebyshev systems see \cite{key-8}):
\begin{defn}
A system of real functions $\left\{ u_{i}(t)\right\} _{0}^{k}$ is
called a Chebyshev system of order k if any linear combination:
\begin{equation}
p(t)=\sum_{i=0}^{k}a_{i}u_{i}(t)\hspace{6em}(\sum_{i=0}^{k}a_{i}^{2}>0)
\end{equation}
has at most k roots.
\end{defn}

One readily sees that $\left\{ u_{i}(t)\right\} _{0}^{k}$ is a Chebyshev
system if and only if the determinant of the matrix $\left\{ u_{i}(t_{j})\right\} _{i\times j}$
doesn't vanish for any $t_{0}<\cdots<t_{k}\in\mathbb{R}$. This in
turn is equivalent to the existence of a linear combination of the
form (2.1) whose roots are exactly $t_{1},\ldots,t_{k}.$ For a fixed
set of numbers $p_{0}<\cdots<p_{k}$ where $p_{0}=0,$ $p_{1}=1$
we define:
\begin{equation}
u_{i}(t)=\begin{cases}
\lvert t\rvert^{p_{i}}sgn(t) & i\ \text{is odd}\\
\lvert t\rvert^{p_{i}} & i\ \text{is even}
\end{cases}.
\end{equation}

\begin{lem}
For $k<5$, $\left\{ u_{i}(t)\right\} _{0}^{k}$ is a Chebyshev system.
\end{lem}

\begin{proof}
We shall prove the lemma for $k=4$ (the other cases can be proved
in a similar yet easier way). Let $p(t)$ be of the form (2.1), we
may assume $a_{4}=1$ since if $a_{4}=0$ we get the case $k=3.$
By Rolle's theorem it is enough to show $p'(t)$ has at most 3 roots.
Define:
\[
q(t)=a_{2}u'_{2}(t)+a_{3}u'_{3}(t)+u'_{4}(t).
\]
\begin{center}
\newcommand{\image}[1]{\includegraphics[width=0.7\linewidth]{#1}}
    \image{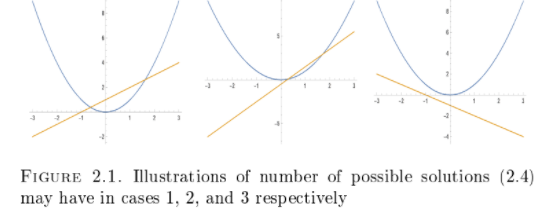}
\end{center}
It is enough to show that there exists a line segment $[b,c]\subset\mathbb{R}$
such that $q$ is monotonically decreasing in $[b,c]$ and monotonically
increasing both in $[c,\infty)$ and in $(-\infty,b]$. For $t>0$
we have:
\begin{equation}
q'(t)=\tilde{a}_{2}t^{p_{2}-2}+\tilde{a}_{3}t^{p_{3}-2}+\tilde{a}_{4}t^{p_{4}-2}
\end{equation}
\[
\tilde{a_{i}}=a_{i}p_{i}(p_{i}-1).
\]
 Notice that $a_{i}\text{ and }\tilde{a_{i}}$ have the same sign.
Again by normalizing we may assume $\tilde{a}_{4}=1$. We now have
$q'(t)=0$ if and only if:
\[
\tilde{a}_{2}+\tilde{a}_{3}t^{p_{3}-p_{2}}+t^{p_{4}-p_{2}}=0.
\]
Changing variables to $x=t^{p_{3}-p_{2}}$ and defining $\beta=\frac{p_{4}-p_{2}}{p_{3}-p_{2}}>1$
we get:
\begin{equation}
\tilde{a}_{2}+\tilde{a}_{3}x+x^{\beta}=0.
\end{equation}

Since $x^{\beta}$ is strictly convex, it intersects any affine function
at most twice. Further inspection shows that the number of positive
roots of (2.4) depends on $a_{2},a_{3}$ in the following way (depicted
in Figure 2.1):
\begin{enumerate}
\item If $a_{2}\leq0$, then (2.4) has exactly one positive solution, since
the function $\lvert x\rvert^{\beta}$ is strictly convex in all of
$\mathbb{R}$, but the line $-\tilde{a}_{2}-\tilde{a}_{3}x$ must
intersect it once in $(0,\infty)$ and one in $(-\infty,0]$.\medskip{}
\item If $a_{2}>0$, $a_{3}\leq0$, then (2.4) has at most two positive
solutions.\medskip{}
\item If $a_{2}>0$,$a_{3}>0$, then (2.4) has no positive solutions, since
in this case, the left-hand side of (2.4) is positive.
\end{enumerate}

We now examine $t<0$. Setting $s=-t>0$ we get:
\begin{equation}
q'(s)=\tilde{a}_{2}s^{p_{2}-2}-\tilde{a}_{3}s^{p_{3}-2}+\tilde{a}_{4}s^{p_{4}-2}.
\end{equation}
But this is just (2.3) with the sign of $a_{3}$ changed. Thus, if
case 1 holds, for $t>0$, $q(t)$ can only decrease up to some point
$c\geq0$ and then must increase for $t>c$ (the limit of $q$ at
infinity is infinity). For $t<0,$ $q(t)$ must increase up to some
point $b\leq0$ and may only decrease when $t>b$. Thus the claim
of the lemma holds for case 1. If case 2 holds, for $t>0$, $q(t)$
may change its monotonicity at most twice, but for $t<0$, $q(t)$
must be monotonically increasing (since for $t<0$ when checking our
list of possible roots of the derivative, we change the sign of $a_{3}$).
Thus the claim of the lemma still holds. Case 3 is just a reflection
of case 2 and so the lemma holds for all cases.
\end{proof}
Now let $\theta_{1},\theta_{2}$ be two outer normal unit vectors
to the centered simplex $\Delta_{n}$ and recall that $\Gamma_{n}\sim Uni(\Delta_{n}).$
Define:
\begin{equation}
\Gamma_{n}^{s}=\frac{s\theta_{1}-(1-s)\theta_{2}}{\lvert s\theta_{1}-(1-s)\theta_{2}\rvert}\cdot\Gamma_{n}\hspace{4em}0\leq s\leq1
\end{equation}
 and let $g_{n}^{s}(t)$ be the density of $\Gamma_{n}^{s}$. We now
state a few important properties of $g_{n}^{s}$:
\begin{lem}
The following properties hold for $g_{n}^{s}$:
\begin{enumerate}
\item $g_{n}^{0}(x)=g_{n}^{1}(-x).$\medskip{}
\item $\left(g_{n}^{0}(x)\right)^{\frac{1}{n-1}}$ is affine on its support.\medskip{}
\item For $0<s<1,$ let $\text{Supp }g_{n}^{s}=[a_{s},b_{s}],$ then $\exists c_{s}\in[a_{s},b_{s}]$
such that $\left(g_{n}^{s}(x)\right)^{\frac{1}{n-1}}$ is affine in
$[a_{s},c_{s}]$ and in $[c_{s},b_{s}].$ In addition, $g_{n}^{s}$
is continuous.
\end{enumerate}
\end{lem}

\begin{proof}
(1) follows from the symmetry of the simplex. (2) follows from the
equality case of Brunn's principle. To see (3) holds notice that we
can split $\Delta_{n}$ into two simplices whose intersection is a
face orthogonal to the unit vector $\frac{s\theta_{1}-(1-s)\theta_{2}}{\lvert s\theta_{1}-(1-s)\theta_{2}\rvert}$.
Thus, using the equality case of Brunn's principle once again, we
see that the claim holds (see e.g. \cite{key-9} for more details).
\end{proof}
Before proving Theorem 1, we give a similar result with regard to
the following auxiliary operator:
\begin{equation}
\alpha^{*}(K):=\max_{\theta\in S^{n-1}}\frac{\lvert\mathbb{E}[\lvert X\cdot\theta\rvert^{q}sgn(x)]\rvert^{\frac{1}{q}}}{\lVert X\cdot\theta\rVert_{p}}.
\end{equation}
This operator can be thought of as measuring the largest asymmetry
between tails of the marginals of K.
\begin{thm}
For any centered convex body $K\subset\mathbb{R}^{n}$:
\begin{equation}
\alpha^{*}(K)\leq\alpha^{*}(\Delta_{n}),
\end{equation}
where equality holds if and only if K is a cone. In addition, the
unit vectors in which the maximum is achieved are normal to (one of
the) the base(s) of the cone.
\end{thm}

\begin{proof}
As before, for some $\theta\in S^{n-1}$, let $f_{\theta}(x)$ be
the density of $X\cdot\theta.$ We shall also abbreviate $g_{n}^{0}$
by $g_{n}$ for the duration of this proof. By the homogeneity of
(2.8) we may assume:
\begin{equation}
\int\rvert x\rvert^{p}f_{\theta}(x)dx=\int\rvert x\rvert^{p}g_{n}(x)dx=1.
\end{equation}
Assume by contradiction that:
\begin{equation}
\int\rvert x\rvert^{q}sgn(x)f_{\theta}(x)dx>\int\rvert x\rvert^{q}sgn(x)g_{n}(x)dx.
\end{equation}
Our general idea is to use Lemma 8 to interlace the difference function
$h:=g_{n}-f_{\theta}$ with appropriate linear combinations of $\{1,x,\lvert x\rvert^{p},\lvert x\rvert^{q}sgn(x)\}$.
We then use four linear constraints (i.e. (2.9) (2.10) and the fact
that $f_{\theta}$ and $g_{n}$ are centered densities), to force
$h$ to change sign at least four times. This will contradict Lemmas
6 and 9, which state that $\left(g_{n}\right)^{\frac{1}{n-1}}$ is
non-negative, affine on $[a_{0},b_{0}]$ and continuous at every point
but $a_{0},$ and $\left(f_{\theta}\right)^{\frac{1}{n-1}}$ is non-negative
and concave on its support. Thus $h$ may change sign at most twice
at $(a_{1},\infty)$ and not at all at $(-\infty,a_{1})$, making
for a maximum of three possible sign changes. We now show (2.10) implies
$h$ must change sign at least four times. First, since $f_{\theta}$
and $g_{n}$ are densities, we get:
\[
\int h(x)dx=0.
\]
Thus h must change sign at least once. Now if $h$ changes sign only
once, say at a point $x_{1}\in\mathbb{R}$, then $(x-x_{1})h(x)$
never changes sign, but since $f_{\theta}$ and $g_{n}$ are centered
densities, we get:
\[
\int(x-x_{1})h(x)dx=0,
\]
 which contradicts (2.10). Assume now $h$ changes sign exactly twice
at points $x_{1},\ x_{2}$. Then by Lemma 8 there exists $p(x)=d_{1}+d_{2}x+\lvert x\rvert^{p}$
whose roots are exactly $x_{1},x_{2}$ and so $p(x)h(x)$ never changes
sign. Using (2.9) we now get:
\[
\int p(x)h(x)dx=0,
\]
which contradicts (2.10). Finally, assume $h$ changes sign exactly
three times at points $x_{1},x_{2},x_{3}$, again by Lemma 8 there
exists $p(x)=d_{1}+d_{2}x+d_{3}\lvert x\rvert^{p}+\lvert x\rvert^{q}sgn(x),$
whose roots are exactly $x_{1},x_{2},x_{3}$. Also note that in this
case $f_{\theta}$ and $g_{n}$ must intersect twice on $(a_{1},\infty)$
and so $h(x)$ must be non-negative for large values of $x$, thus
$h(x)p(x)\geq0$ for all $x\in\mathbb{R}$. However (2.9) and (2.10)
give us:
\[
\int p(x)h(x)dx<0,
\]
which is a contradiction. We thus showed:
\[
\frac{\mathbb{E}[\lvert X\cdot\theta\rvert^{q}sgn(X\cdot\theta)]}{\lVert X\cdot\theta\rVert_{p}}\leq\frac{\mathbb{E}[\lvert\Gamma_{n}^{1}\rvert^{q}sgn(\Gamma_{n}^{1})]}{\lVert\Gamma_{n}^{1}\rVert_{p}}
\]
and that equality implies that $X\cdot\theta\text{ coincides with }\Gamma_{n}^{1}$
in law, which in turn implies K is a cone and $\theta$ is normal
to $K$'s base by the equality case of Brunn's principle.
\end{proof}
\begin{proof}[Proof of Theorem 1]
\noindent  The proof is very similar to that of Theorem 10. Using
the same notations as before we can again assume by homogeneity :
\begin{equation}
\int\rvert x\rvert^{p}f_{\theta}(x)dx=\int\rvert x\rvert^{p}g_{n}^{s}(x)dx=1\hspace{4em}\forall0\leq s\leq1
\end{equation}
and assume by contradiction:
\begin{equation}
\int\rvert x\rvert^{q}f_{\theta}(x)dx>\int\rvert x\rvert^{q}g_{n}^{s}(x)dx\hspace{4em}\forall0\leq s\leq1.
\end{equation}
Now choose some $p<r<q$. Theorem 10 tells us:
\[
\int\rvert x\rvert^{r}sgn(x)g_{n}^{1}(x)dx\leq\int\rvert x\rvert^{r}sgn(x)f_{\theta}(x)dx\leq\int\rvert x\rvert^{r}sgn(x)g_{n}^{o}(x)dx,
\]
where the first inequality just follows from property (1) in Lemma
9. Thus from continuity we may choose some $0\leq s_{0}\leq1$ such
that:
\begin{equation}
\int\rvert x\rvert^{r}sgn(x)f_{\theta}(x)dx=\int\rvert x\rvert^{r}sgn(x)g_{n}^{s_{0}}(x)dx.
\end{equation}
We now proceed exactly as we did before; by Lemmas 6 and 9 the difference
function $h=g_{n}^{s_{0}}-f_{\theta}$ can change sign at most 4 times.
The same arguments as before show that $h$ must change sign at least
4 times, but if $h$ changes sign exactly 4 times at points $x_{1},\ldots,x_{4}$,
again by Lemma 8 there exists $p(x)=d_{1}+d_{2}x+d_{3}\lvert x\rvert^{p}+d_{4}\lvert x\rvert^{r}sgn(x)+\lvert x\rvert^{q}$,
whose roots are exactly $x_{1},\ldots,x_{4}$. Since $g_{n}^{s_{0}}$
and $f_{\theta}$ must intersect twice in $[c_{s_{0}},b_{s_{0}}]$
we have $h(x)\geq0$ for large values of $x$ and so $p(x)h(x)\geq0$.
But (2.11) (2.12) and the fact $g_{n}^{s_{0}}$ and $f_{\theta}$
are centered densities give us:
\[
\int p(x)h(x)dx<0,
\]
which is a contradiction. We thus showed:
\[
\frac{\lVert X\cdot\theta\rVert_{q}}{\lVert X\cdot\theta\rVert_{p}}\leq\frac{\lVert\Gamma_{n}\cdot\theta_{s_{0}}\rVert_{q}}{\lVert\Gamma_{n}\cdot\theta_{s_{0}}\rVert_{p}}
\]
 and so:
\[
\alpha(K)\leq\alpha(\Delta_{n}).
\]
\renewcommand{\qedsymbol}{$\blacksquare$}
\end{proof}
As mentioned in the Introduction, the proof of Theorem 1 shows that
the marginal for which the maximum norm ratio in the simplex (and
in any other convex body) is achieved belongs to the 1-parameter family
$\Gamma_{n}^{s}$ and thus any convex body which maximizes $\alpha$
can be split into two cones by some hyperplane. Quite unexpectedly,
calculations show that for different values of $p$ and $q$, the
maximum norm ratio is achieved in different values of $s$.

Notice also that all we needed for the proof of Theorem 1 was the
Chebyshev system property of $\{1,x,\lvert x\rvert^{p},\lvert x\rvert^{r}sgn(x),\lvert x\rvert^{q}\}$.
Thus, using other Chebyshev systems might bring forth new inequalities.
Let us give one more example of a Chebyshev system:
\begin{lem}
For any $k\in\mathbb{N},$ let $\phi\in C^{k}(\mathbb{R})$, if $\phi^{(k)}>0$
the set $\{1,x,...,x^{k-1},\phi\}$ is a Chebyshev system.
\end{lem}

\begin{proof}
We use Rolle's Theorem iteratively. If for some $p(x)=\sum_{i=0}^{k-1}d_{i}x^{i}+d_{k}\phi(x)$
there are more then $k$ roots, then $p^{(k)}(x)=\phi^{(k)}(x)$ has
at least one root, which is, of course, a contradiction (this proof
is taken from \cite[Chapter 2, Example E]{key-10}).
\end{proof}
\begin{proof}[Proof of Theorem 3.]
 The proofs when $k=3,4$ are identical to those of Theorems 10 and
1 respectively.

\renewcommand{\qedsymbol}{$\blacksquare$}
\end{proof}

\section{the dimension independent case}

Looking at the log-concave case, we notice the role of our one-parameter
family $\Gamma_{n}^{s}$ is taken by following simpler family of random
variables:
\begin{equation}
\Gamma^{s}=s\Gamma-(1-s)\Gamma'.
\end{equation}
 Here, $\Gamma,\Gamma'$ are i.i.d with density $g(x)=e^{-(x+1)}1_{[-1,\infty)}$.
Indeed, replacing the power of $\frac{1}{n-1}$ by log, Lemma 6 becomes
the definition of a log-concave random variable, and Lemma 9 is easy
to verify for $g^{s}(x)$, the densities of $\Gamma^{s}$. Thus, one
can reformulate theorems 1,3 and 10 for the log-concave case and prove
them in the exact same way:
\begin{thm}
Let X be a centered log concave random variable, then:
\begin{equation}
\frac{\lVert X\rVert_{q}}{\lVert X\rVert_{p}}\leq\max_{0\leq s\leq1}\frac{\lVert\Gamma^{s}\rVert_{q}}{\lVert\Gamma^{s}\rVert_{p}}.
\end{equation}
\end{thm}

\begin{thm}
Let X be a log concave random variable in isotropic position and let
$\phi\in C^{k}(\mathbb{R})$ and $\phi^{(k)}(x)>0,\ \forall x\in\mathbb{R}$
where $k\in\{3,4\}.$ We then have:
\begin{equation}
\mathbb{E}[\phi(X)]\leq\max_{0\leq s\leq1}\mathbb{E}[\phi(\Gamma^{s})].
\end{equation}
If $k=3$ the maximum is achieved at $s=0$ and equality implies $X$
coincides with $\Gamma$ in law.
\end{thm}

\begin{thm}
Let X be a centered log concave random variable, then:
\begin{equation}
\frac{\mathbb{E}[\lvert X\rvert^{q}sgn(X)]}{\lVert X\rVert_{p}}\leq\frac{\mathbb{E}[\lvert\Gamma\rvert^{q}sgn(\Gamma)]}{\lVert\Gamma\rVert_{p}},
\end{equation}
with equality if and only if $X$ coincides with $\Gamma$ in law.
\end{thm}

Note that Theorem 5 follows from Theorem 14. The following question
naturally arises: which is the maximizer of the norm ratio among $\Gamma^{s}$
for $0\leq s\leq1$? We shall use the new convenient formula (3.1)
to answer this question when $p=2$ and $q$ is an even integer, but
first, let us recall that the cumulants $\left\{ k_{n}(X)\right\} _{n=0}^{\infty}$
of a random variable $X$ are defined by:
\begin{equation}
log\left(\mathbb{E}[e^{tX}]\right)=\sum_{n=0}^{\infty}\frac{k_{n}(X)}{n!}t^{n}.
\end{equation}
See e.g. \cite{key-7} for proof that when $X$ is log-concave $\mathbb{E}[e^{tX}]$
is analytic in a small neighborhood of zero. Let us also define $\mu_{n}(X)$
to be the n-th moment of $X$. Two useful properties of the cumulants
of a random variable are the following:
\begin{enumerate}
\item For any $X,Y$ independent random variables, $a,b\in\mathbb{R}$ and
$n\in\mathbb{N},$ we have:
\begin{equation}
k_{n}(aX+bY)=a^{n}k_{n}(X)+b^{n}k_{n}(Y).
\end{equation}
\item Knowing the value of the cumulants of $X$ (and assuming again that
X is centered) we can restore the value of its moments using the following
formula:
\begin{equation}
\mu_{n}(X)=\sum_{i=1}^{n}\binom{n-1}{i-1}k_{i}(X)\mu_{n-i}(X).
\end{equation}
\end{enumerate}
We now calculate the cumulants of $\Gamma$:
\[
log\left(\mathbb{E}[e^{t\Gamma}]\right)=log\left(\int_{-1}^{\infty}e^{-x-1}\cdot e^{tx}dx\right)=log\left(\frac{e^{-t}}{1-t}\right)=-t-log(1-t)=\sum_{n=2}^{\infty}\frac{t^{n}}{n},
\]
and so:
\begin{equation}
k_{0}(\Gamma)=k_{1}(\Gamma)=0,\ \text{and }k_{n}(\Gamma)=\left(n-1\right)!\hspace{3em}n\geq2.
\end{equation}

\begin{proof}[Proof of Theorem 4]
 Reparametrizing (3.2) in Theorem 12, we get:
\begin{equation}
\left(\frac{\lVert X\rVert_{n}}{\lVert X\lVert_{2}}\right)^{n}\leq\max_{0\leq t\leq\frac{\pi}{2}}\mu_{n}\left(\cos(t)\Gamma-\sin(t)\Gamma'\right),
\end{equation}
with equality if and only if $\exists0\leq t_{0}\leq\frac{\pi}{2}$
such that:
\begin{equation}
\frac{X}{\lVert X\rVert_{2}}=\cos(t_{0})\Gamma-\sin(t_{0})\Gamma'.
\end{equation}
Thus, the only remaining thing to prove is:
\[
\mu_{n}\left(\cos(t)\Gamma-\sin(t)\Gamma'\right)<\mu_{n}(\Gamma)\hspace{4em}\forall0<t<\frac{\pi}{2}.
\]
From the discussion above we see that:
\begin{eqnarray}
\lvert k_{n}\left(\cos(t)\Gamma-\sin(t)\Gamma'\right)\rvert & = & \lvert\cos^{n}(t)+(-\sin(t))^{n}\rvert(n-1)!\nonumber \\
 & < & (n-1)!=\lvert k_{n}(\Gamma)\rvert\hspace{6em}\forall n\geq2,\ 0<t<\frac{\pi}{2}.
\end{eqnarray}
Here, the strict inequality comes from the classic norm inequality
$\lVert x\rVert_{n}<\lVert x\rVert_{2}$ for any vector $x$ with
at least two non zero coordinates. We now assume by induction that:
\[
\lvert\mu_{i}\left(\cos(t)\Gamma-\sin(t)\Gamma'\right)\rvert<\lvert\mu_{i}(\Gamma)\rvert\hspace{3em}\forall2<i<n,\ 0<t<\frac{\pi}{2}.
\]
We thus have:
\begin{align*}
\lvert\mu_{n}\left(\cos(t)\Gamma-\sin(t)\Gamma'\right)\rvert & \leq\sum_{i=1}^{n}\binom{n-1}{i-1}\lvert k_{i}(\cos(t)\Gamma-\sin(t)\Gamma')\rvert\lvert\mu_{n-i}(\cos(t)\Gamma-\sin(t)\Gamma')\rvert\\
 & <\sum_{i=1}^{n}\binom{n-1}{i-1}\lvert k_{i}(\Gamma)\rvert\lvert\mu_{n-i}(\Gamma)\rvert=\lvert\mu_{n}(\Gamma)\rvert.
\end{align*}
This completes our proof.

\renewcommand{\qedsymbol}{$\blacksquare$}
\end{proof}
Theorem 12 reduces the problem of finding the maximum norm ratio over
all log-concave random variables into finding the maximum norm ratio
over $\Gamma^{s}$. If $p$ and $q$ are even integers, this problem
is equivalent to finding:
\begin{equation}
\max_{0\leq s\leq1}\frac{\left(r_{q}(s)\right)^{p}}{\left(r_{p}(s)\right)^{q}},
\end{equation}
where $r_{p}(s)=\mathbb{E}[s\Gamma+(s-1)\Gamma']^{p}$ is a family
of polynomials. We now present a technique, which worked for all even
integers up to 100, to prove the maximum is achieved exactly at $s=0$
and $s=1$.
\begin{thm}
Let X be a centered log concave random variable and let $p,\ q$ be
even numbers such that $p<q<100$ then:
\begin{equation}
\frac{\lVert X\rVert_{q}}{\lVert X\rVert_{p}}\leq\frac{\lVert\Gamma\rVert_{q}}{\lVert\Gamma\rVert_{p}}=\frac{\left(!q\right)^{\frac{1}{q}}}{\left(!p\right)^{\frac{1}{p}}},
\end{equation}
where equality holds if and only if $X$ coincides with $\pm\Gamma$
in law.
\end{thm}

\begin{proof}
Unfortunately, this proof is computer-assisted. First notice that
the rational function in (3.12) is symmetric with respect to $\frac{1}{2}$,
and that it is enough to prove our hypothesis for $p=q-2.$ Thus our
goal will be to show:
\[
\frac{d}{ds}log\left(\frac{\left(r_{q}(s)\right)^{q-2}}{\left(r_{q-2}(s)\right)^{q}}\right)\leq0\hspace{5em}\forall0\leq s\leq\frac{1}{2},
\]
or equivalently:
\begin{equation}
h_{q}(s):=(q-2)\cdot r'_{q}(s)\cdot r_{q-2}(s)-q\cdot r'_{q-2}(s)\cdot r_{q}(s)\leq0\hspace{3em}\forall0\leq s\leq\frac{1}{2}.
\end{equation}
Define now:
\[
\tilde{h}_{q}(s)=\frac{h_{q}(s-\frac{1}{2})}{s(s-\frac{1}{2})(s+\frac{1}{2})}=\sum a_{i,q}s^{2i}.
\]
Calculating the coefficients of the first few polynomials we get:
\begin{align*}
\tilde{h}_{4}(s) & =-96\\
\tilde{h}_{6}(s) & =-720(15+8s^{2}+144s^{4})\\
\tilde{h}_{8}(s) & =-1680(1485-2880s^{2}+105696s^{4}+104448s^{6}+268544s^{8})\\
\tilde{h}_{10}(s) & =-5040(269325-1323000s^{2}+72560880s^{4}+280339200s^{6}\\
 & +1409629440s^{8}+1162622976s^{10}+1050406912s^{12}).
\end{align*}
One can easily check that when $q\leq10,$ $a_{i.q}\leq0$ for all
$i\neq1$ and $a_{1}^{2}-4a_{0}a_{2}<0.$ Thus $\tilde{h}_{q}$ never
changes sign and $h_{q}(s)$ has the same sign as $s(s-1)(s+1)$,
which proves the theorem. Using computer assistance one can see the
same proof still holds for $q<100$. For details see \cite{Eitan}.
\end{proof}
We finish this paper with a conjecture which generalizes these empiric
results:
\begin{conjecture}
Let X be a centered log-concave random variable. Then for even integers
$p<q$, we have:
\begin{equation}
\frac{\lVert X\rVert_{q}}{\lVert X\lVert_{p}}\leq\frac{\lVert\Gamma\rVert_{q}}{\lVert\Gamma\rVert_{p}}=\frac{\left(!q\right)^{\frac{1}{q}}}{\left(!p\right)^{\frac{1}{p}}},
\end{equation}
with equality if and only if  $X$ coinsides with $\pm\Gamma$ in
law, where $\Gamma$ is a random variable with density $g(x)=e^{-(x+1)}1_{[-1,\infty)}$.

\renewcommand{\partname}{Chapter}
\end{conjecture}

{\normalsize 
\noindent Department of Mathematics, Weizmann Institute of Science, Rehovot, Israel. \\
{\it e-mail:} \verb"yam.eitan@weizmann.ac.il"}


\begin{thebibliography}{10}
\bibitem{key-1}Berwald, L. (1947). Verallgemeinerung eines Mittelwertsatzes
von J. Favard fur positive konkave Funktionen. \textit{Acta Mathematica,
79(1)}, 17-37.

\bibitem{key-2}Borell, C. (1973). Complements of Lyapunov's inequality.\textit{
Mathematische Annalen, 205(4)}, 323-331.

\bibitem{key-3}Brazitikos, S., Giannopoulos, A., Valettas, P.,
\& Vritsiou, B. H. (2014). \textit{Geometry of isotropic convex bodies}
(Vol. 196). American Mathematical Soc.

\bibitem{key-4}Bubeck, S., \& Eldan, R. (2014). The entropic barrier:
a simple and optimal universal self-concordant barrier.\textit{ arXiv
preprint arXiv:1412.1587}.

\bibitem{Eitan} Eitan, Y. (2021). The centered convex body whose marginals have the heaviest tails.
M. Sc. Thesis, Weizmann Institute of Science. Available under: \\
\verb"https://www.weizmann.ac.il/math/klartag/sites/math.klartag/files/uploads/yam.pdf"

\bibitem{key-5}Eskenazis, A., Nayar, P., \& Tkocz, T. (2018).
Sharp comparison of moments and the log-concave moment problem. \textit{Advances
in Mathematics, 334}, 389-416.

\bibitem{key-6}Haagerup, U. (1981). The best constants in the
Khintchine inequality. \textit{Studia Mathematica, 70}, 231-283.


\bibitem{key-7}Klartag, B., \& Milman, E. (2012). Centroid bodies
and the logarithmic Laplace transform--a unified approach. \textit{Journal
of Functional Analysis, 262(1)}, 10-34.

\bibitem{key-8}Kre\u{\i}n, M. G., Nudel, A. A., \& Louvish,
D. (1977). \textit{The Markov moment problem and extremal problems}.
American Mathematical Society.



\bibitem{key-9}Makai, E., \& Martini, H. (1996). The cross-section
body, plane sections of convex bodies and approximation of convex
bodies, I. \textit{Geometriae Dedicata, 63(3)}, 267-296.


\bibitem{key-10}Milman, V. D., \& Pajor, A. (1989). Isotropic
position and inertia ellipsoids and zonoids of the unit ball of a
normed n-dimensional space. In \textit{Geometric aspects of functional
analysis} (pp. 64-104). Springer, Berlin, Heidelberg.

\bibitem{key-11}Nayar, P., \& Oleszkiewicz, K. (2012). Khinchine
type inequalities with optimal constants via ultra log-concavity.
\textit{Positivity, 16(2)}, 359-371.

\bibitem{key-12}Nazarov, F. L., \& Podkorytov, A. N. (2000).
Ball, Haagerup, and distribution functions. In \textit{Complex analysis,
operators, and related topics} (pp. 247-267). Birkhauser, Basel.

\bibitem{key-13}Rogers, C. A., \& Shephard, G. C. (1957). The
difference body of a convex body. \textit{Archiv der Mathematik, 8(3)},
220-233.

\bibitem{key-14}Schneider, R. (2014). \textit{Convex bodies:
the Brunn--Minkowski theory} (No. 151). Cambridge university press.
\end{thebibliography}
\end{document}